\documentclass[12pt]{article}
\usepackage{amsmath,amssymb,amsfonts,amsthm,indentfirst,setspace}
\usepackage{latexsym,bm}
\usepackage{mathrsfs,amsmath,amssymb,amsthm}

\numberwithin{equation}{section}

\theoremstyle{plain}
\newtheorem{theorem}{Theorem}[section]

\newtheorem{lemma}{Lemma}[section]
\newtheorem{claim}{Claim}[section]
\newtheorem{corollary}{Corollary}[section]

\newtheorem{problem}{Problem}[section]

\newcommand{\be}{\begin{equation}}
\newcommand{\ee}{\end{equation}}
\newcommand{\bea}{\begin{eqnarray}}
\newcommand{\eea}{\end{eqnarray}}
\newcommand{\eeas}{\end{eqnarray*}}
\newcommand{\beas}{\begin{eqnarray*}}
\usepackage{amscd}

\numberwithin{equation}{section}

\makeatletter
\newcommand{\trace}{\mathop{\operator@font Trace}}
\newcommand{\vspan}{\mathop{\operator@font Span}}
\newcommand{\Int}{\mathop{\operator@font Int}}
\newcommand{\grad}{\mathop{\operator@font grad}}
\newcommand{\diver}{\mathop{\operator@font div}}
\newcommand{\Ad}{\mathop{\operator@font Ad}}
\newcommand{\id}{\mathop{\operator@font id}}

\newcommand{\spec}{{\mathop{\operator@font Spec}}(\mathcal H)}

\makeatother

\begin{document}

\title{Real hypersurfaces in complex Grassmannians of rank two  with semi-parallel structure Jacobi operator}


\author{Avik De, Tee-How {Loo} and Changhwa Woo\\
Department of Mathematical and Actuarial Sciences\\
Universiti Tunku Abdul Rahman\\
Jalan Sungai Long\\
43000 Cheras, Malaysia. \\
Institute of Mathematical Sciences \\
University of Malaya \\
50603 Kuala Lumpur, Malaysia.\\
Department of Mathematics Education \\
Woosuk University\\
565-701 Wanju, Jeonbuk, Republic Of Korea.	\\
\ttfamily{de.math@gmail.com}, 
\ttfamily{looth@um.edu.my},
\ttfamily{legalgwch@naver.com}
}

\date{}
\maketitle

\abstract{
We prove that there does not exist any real hypersurface in complex Grassmannians of rank two  with semi-parallel structure Jacobi operator. With this result, the non-existence of real hypersurface
in complex Grassmannians of rank two with recurrent structure Jacobi operator is proved.
}




\section{Introduction}
Let  $\hat M^m(c)$ be the compact complex Grassmannian $SU_{m+2}/S(U_2U_m)$ of rank two
 (resp.  noncompact complex Grassmannian $SU_{2,m}/S(U_2U_m)$ of  rank two) for $c>0$ (resp. $c<0$),
where $c=\max||K||/8$ is a scaling factor for the Riemannian metric $g$ and $K$ is the sectional curvature for $\hat M^m(c)$.
 It is an irreducible Riemannian symmetric space equipped with  a K\"ahler structure $J$ and a quaternionic K\"ahler structure $\mathfrak{J}$ not containing $J$.

Let $M$ be a connected, oriented real hypersurface isometrically immersed in $\hat M^m(c)$, $m\geq 2$,
and $N$ be a unit normal vector field on $M$. 
Denote by the same $g$ the Riemannian metric on $M$.
The Reeb vector field $\xi$ is defined by $\xi = -JN$,  and we define $\xi_a = -J_a N$,  $a\in\{1,2,3\}$, where 
$\{J_1,J_2,J_3\}$  is a canonical local basis of $\mathfrak J$.
Denote by $D^\perp$ (resp. $\mathfrak D^\perp$) the distribution on $M$ spanned by $\xi$ (resp. $\{\xi_1,\xi_2,\xi_3\}$).
A real hypersurface $M$ in a K\"ahler  manifold  is said to be Hopf if  the Reeb vector field is principal, that is, $A\xi=\alpha\xi$.

The study of real hypersurfaces in $\hat M^m(c)$ was initiated by  Berndt and Suh in \cite{berndt-suh, berndt-suh2}.
They considered the invariance of   $\mathfrak D^\perp$  under the shape operator $A$ of  Hopf hypersurfaces $M$ in $\hat M^m(c)$
and proved a classification of such Hopf hypersurfaces in $\hat M^m(c)$. 

The structures $J$ and $\mathfrak{J}$ of the ambient space impose several restrictions on the geometry of its real hypersurfaces, for example, there does not exist any semi-parallel real hypersurface in $SU_{m+2}/S(U_2U_m)$
\cite{looth} while the non-existence problem of Hopf hypersurfaces in $\hat M^m(c)$ with parallel Ricci tensor were studied in \cite{suh2, suh3}. 

Besides the shape operator and the Ricci tensor, there are particularly two operators  on a real hypersurface $M$ which draw much attention, 
namely the normal Jacobi operator $R_N$ 
and the structure Jacobi operator $R_\xi$. 
Denote by  $\hat R$ and $R$ the curvature tensor on $\hat M^m(c)$ and that induced on $M$ respectively. 
We define $R_N(X)=\hat R(X,N)N$ and $R_\xi(X)=R(X,\xi)\xi$ for any vector field $X$ tangent to $M$.

A $(1,s)$-tensor field $P$ is said to be \emph{semi-parallel} if $R\cdot P=0$, where the curvature tensor $R$ acts on $P$ as a derivation.
More precisely,
\begin{align*}
(R(X,Y)	&\cdot P)(X_1,\cdots,X_s)\\
							&=R(X,Y)P(X_1,\cdots,X_s)-\sum^s_{j=1}P(X_1,\cdots,R(X,Y)X_j,\cdots,X_s).
\end{align*}
The tensor field $P$  is said to be \emph{recurrent} if there exists an $1$-form $\omega$ on $M$ such that 
\[
(\nabla_XP)(X_1,\cdots,X_s)=\omega(X)P(X_1,\cdots,X_s).
\]
Clearly, a vanishing $\omega$ leads to  parallelism of $P$.

Recently, we proved the non-existence of real hypersurface in $SU_{m+2}/S(U_2U_m)$,  $m \ge 3$,  with  pseudo-parallel normal  Jacobi operator
\cite{loothfilo}.
 On the other hand, related to the structure Jacobi operator $R_\xi$, 
Jeong, et al. proved that there does not exist any Hopf  hypersurface 
in $SU_{m+2}/S(U_2U_m)$,  $m \ge 3$, with parallel structure Jacobi operator \cite{jps}. Also, the non-existence of  Hopf hypersurfaces with $D^\perp$-parallel structure Jacobi operator is obtained under certain conditions \cite{jmps}. 
Jeong,   et al.
considered Reeb-parallel structure Jacobi operator and proved the following:
\begin{theorem}[\cite{jks}]
	Let $M$ be a Hopf hypersurface in 
	$SU_{m+2}/S(U_2U_m)$,  $m \ge 3$, with Reeb parallel structure Jacobi operator. If the principal curvature of the Reeb vector field $\xi$ on $M$ is non-vanishing and constant along the direction
	of the Reeb vector field $\xi$, then $M$ is an open part of a tube around a totally geodesic
	$SU_{m+1}/S(U_2U_{m-1})$
	in 
	$SU_{m+2}/S(U_2U_m)$
	 with radius $r\in (0,\frac{\pi}{4\sqrt{2}}) \cup (\frac{\pi}{4\sqrt{2}}, \frac{\pi}{\sqrt{8}})$.
\end{theorem}
We say that a real hypersurface $M$ has commuting structure Jacobi operator if it commutes with any other Jacobi operator defined on $M$, that is, $R_\xi \cdot R_X = R_X \cdot R_\xi$ for any $X$ tangent to $M$. Machado,  et al. proved the non-existence of Hopf real hypersurface in $SU_{m+2}/S(U_2U_m)$,  $m \ge 3$,  with commuting structure Jacobi operator under certain conditions \cite{mps}. They also classified real hypersurfaces in 
$SU_{m+2}/S(U_2U_m)$,    $m \ge 3$,
 with $\hat R_N\cdot R_\xi=R_\xi\cdot \hat R_N$. 
In \cite{recurrent}, 
Jeong, et al. proved the following:
\begin{theorem}[\cite{recurrent}]
There does not exist any Hopf hypersurface in 
$SU_{m+2}/S(U_2U_m)$,  $m \ge 3$, with recurrent structure Jacobi operator if the distribution $\mathfrak D$ or 
$\mathfrak D^\perp$-component of the Reeb vector field is invariant under the shape operator.		
\end{theorem}
On the other hand,  under certain restrictions, 
Hwang, et al.  obtained the  following  non-existence result.
\begin{theorem}[\cite{hwang-lee-woo}]
There  does  not  exist  any Hopf  hypersurface in $SU_{m+2}/S(U_2U_m)$,  $m \ge 3$, with semi-parallel structure Jacobi operator if the smooth function $\alpha=g(A\xi, \xi)$ is constant along the direction of $\xi$.	
\end{theorem}
Motivated from the above studies, a natural question arises:
\begin{problem}
Does there exist real hypersurface in $\hat M^m(c)$ with parallel, recurrent or semi-parallel structure Jacobi operator? 
\end{problem}
In the present paper we first prove the following:

\begin{theorem}\label{T2}
There does not exist any connected real hypersurface in $\hat M^m(c)$, $m\geq3$, with semi-parallel structure Jacobi operator.
\end{theorem}

The non-existence of real hypersurfaces with semi-parallel structure Jacobi operator in a non flat complex space form has been proved in  \cite{cho-kimura, ivey-ryan}.
We also remark that Theorem~\ref{T2} holds for non-Hopf real hypersurfaces as well and no further conditions are imposed. 
By a result  in \cite{de-loo}, we learn that if a tensor field is recurrent, it is always semi-parallel. Hence, as a corollary we obtain the following:
\begin{corollary}
	There does not exist any real hypersurface in $\hat M^m(c)$, $m\geq3$, with parallel or recurrent structure Jacobi operator.
\end{corollary}

\section{Preliminaries}
In this section, we recall some  fundamental identities for real hypersurfaces in  complex Grassmannian of rank two,
which  have been proven in 
\cite{berndt-suh, berndt-suh2, 
lee-loo, looth}. 

Let $M$ be a connected, oriented real hypersurface isometrically immersed in $\hat M^m(c)$, $m\geq  3$.
 The almost contact metric 3-structure $(\phi_a,\xi_a,\eta_a,g)$ on $M$ is given by
$$
J_a X=\phi_a X+\eta_a (X) N,\quad\quad J_a N=-\xi_a,\quad\quad \eta_a(X)=g(X,\xi_a),
$$
for any $X\in TM$, where  $\{J_1,J_2,J_3\}$ is a  canonical local basis  of $\mathfrak J$ on $\hat M^m(c)$.
It follows that
\begin{align*}
\phi_a\phi_{a+1}-\xi_a\otimes\eta_{a+1}=\phi_{a+2}
\\
\phi_a\xi_{a+1}=\xi_{a+2}=-\phi_{a+1}\xi_a  
\end{align*}
for $a\in\{1,2,3\}$. The indices in the preceding equations are taken modulo three.

The  K\"ahler structure $J$ induces on $M$ an almost contact metric structure $(\phi,\xi,\eta,g)$, namely,
\begin{align*}	
JX=\phi X+\eta(X)N,	\quad JN=-\xi, \quad \eta(X)=g(X,\xi).
\end{align*}
Let $\mathfrak D^\perp=\mathfrak JTM^\perp$, and $\mathfrak D$ its orthogonal complement in $TM$.
We define a local  $(1,1)$-tensor field $\theta_a$ on $M$  by
\[
\theta_a :=\phi_a\phi -\xi_a\otimes\eta.
\]
Denote by $\nabla$ the Levi-Civita connection on $M$. Then there exist local $1$-forms $q_a$, $a\in\{1,2,3\}$ such that
\begin{eqnarray}\label{eqn:contact}
\left.\begin{aligned}
& \nabla_X \xi = \phi AX \\
&\nabla_X \xi_a = \phi_a AX+q_{a+2}(X)\xi_{a+1}-q_{a+1}(X)\xi_{a+2} \\
&\nabla_X\phi\xi_a
=\theta_aAX+\eta_a(\xi)AX+q_{a+2}(X)\phi\xi_{a+1} - q_{a+1}(X)\phi\xi_{a+2}. 
\end{aligned}\right\}
\end{eqnarray}
The following identities are known.
\begin{lemma}[\cite{lee-loo}]
\label{lem:theta}
\begin{enumerate}
\item[(a)] $\theta_a$ is symmetric,
\item[(b)] $\phi\xi_a=\phi_a\xi$,
\item[(c)] $\theta_a\xi=-\xi_a; \quad \theta_a\xi_a=-\xi; \quad \theta_a\phi\xi_a=\eta(\xi_a)\phi\xi_a$,
\item[(d)]  $\theta_a\xi_{a+1}= \phi\xi_{a+2}=-\theta_{a+1}\xi_a$,
\item[(e)] 
$-\theta_a\phi\xi_{a+1}+\eta(\xi_{a+1})\phi\xi_a =\xi_{a+2}=\theta_{a+1}\phi\xi_a-\eta(\xi_a)\phi\xi_{a+1}$.
\end{enumerate}
\end{lemma}
\begin{lemma}[\cite{lee-loo}]\label{lem:Aphixi_a}
If $\xi\in\mathfrak D$ everywhere, then $A\phi\xi_a=0$ for $a\in\{1,2,3\}$.
\end{lemma}

For each $x\in M$,   we define a subspace $\mathcal H^\perp$ of $T_xM$ by
$$\mathcal H^\perp: =\mathrm{span}\{\xi,\xi_1,\xi_2,\xi_3,\phi\xi_1,\phi\xi_2,\phi\xi_3\}.$$
Let $\mathcal{H}$ be the orthogonal complement of $\mathcal {H}^\perp$ in $T_xM$.
Then
$\dim\mathcal H=4m-4$  (resp. $\dim\mathcal H=4m-8$) when $\xi\in\mathfrak D^\perp$ (reps. $\xi\notin\mathfrak D^\perp$).
Moreover, $\theta_{a|\mathcal{H}}$ has two eigenvalues: $1$ and $-1$.
Denote by  $\mathcal H_a(\varepsilon)$ the eigenspace corresponding to the eigenvalue $\varepsilon$ of
${\theta_a}_{|\mathcal H}$.
Then $\dim \mathcal H_a(1)=\dim \mathcal H_a(-1)$ is even, and
\begin{align*}
\begin{aligned}
&\phi\mathcal H_a(\varepsilon)=\phi_a\mathcal H_a(\varepsilon)=\theta_a\mathcal H_a(\varepsilon)=\mathcal H_a(\varepsilon) \\
&\phi_b\mathcal H_a(\varepsilon)=\theta_b\mathcal H_a(\varepsilon)=\mathcal H_a(-\varepsilon), \quad (a\neq b).
\end{aligned}
\end{align*}

We define the tensor fields $\theta$, $\phi^\perp$, $\xi^\perp$ and $\eta^\perp$  on $M$ as follows
\begin{align*}
\theta:=&\sum^3_{a=1}\eta_a(\xi)\theta_a, \quad \phi^\perp:=\sum^3_{a=1}\eta_a(\xi)\phi_a, \quad
\xi^\perp:=\sum^3_{a=1}\eta_a(\xi)\xi_a, \quad \eta^\perp:=\sum^3_{a=1}\eta_a(\xi)\eta_a.
\end{align*}
Then for each  $x\in M$ with $\xi^\perp\neq0$,  $\theta_{|\mathcal H}$ has two eigenvalues $\varepsilon||\xi^\perp||$, $\varepsilon\in\{1,-1\}$.
Let $\mathcal H(\varepsilon)$ be the eigenspace of $\theta_{|\mathcal H}$ corresponding to  $\varepsilon||\xi^\perp||$. Then
\begin{enumerate}
\item[(a)] $\phi \mathcal H(\varepsilon)=\phi^\perp \mathcal H(\varepsilon)= \mathcal H(\varepsilon)$,
\item[(b)] $\dim \mathcal H(1)=\dim \mathcal H(-1)$ is even.
\end{enumerate}
Moreover,
we can take a canonical local basis of $\mathfrak J$ on a neighborhood $G\subset M$ of
such a point $x$ such that 
\begin{align*}
&\xi_1=\frac{\xi^\perp}{||\xi^\perp||}, \quad  0<\eta_1(\xi)=||\xi^\perp||\leq 1, \quad \eta_2(\xi)=\eta_3(\xi)=0\\
&\mathcal H(\varepsilon)=\mathcal H_1(\varepsilon),  \quad \theta=\eta_1(\xi)\theta_1,   \quad \phi^\perp=\eta_1(\xi)\phi_1,
				\quad \eta^\perp=\eta_1(\xi)\eta_1.
\end{align*}
In particular,  if $||\xi^\perp||=1$ at  $x$, then
\begin{align*}
&\xi_1=\xi=\xi^\perp, \quad  \xi_2=\theta\xi_2=\phi\xi_3, \quad \xi_3=\theta\xi_3=-\phi\xi_2.
\end{align*}
Throughout this paper, we always consider such a local orthonormal frame $\{\xi_1,\xi_2,\xi_3\}$ on $\mathfrak D^\perp$ under these situations.

A straightforward calculation gives
\begin{align}\label{eqn:global}
(\nabla_X\theta)Y=\eta^\perp(\phi Y)AX-g(AX,Y)\phi\xi^\perp
															+2\sum^3_{a=1}\eta_a(\phi AX)\theta_aY.
\end{align}
The equations of Gauss and Codazzi are respectively given by
$$\begin{aligned}
R(X,Y)Z=&g( AY,Z) AX-g( AX,Z) AY+c\{g( Y,Z) X-g(X,Z) Y\\
&+g(\phi Y,Z)\phi X-g(\phi X,Z)\phi Y -2g(\phi X,Y)\phi Z\}\\
&+c\sum_{a=1}^3\{g(\phi_aY,Z)\phi_aX-g(\phi_aX,Z) \phi_aY-2g(\phi_aX,Y)\phi_aZ\\
&+g(\theta_aY,Z)\theta_aX-g(\theta_aX,Z)\theta_aY\}
\end{aligned}$$
\begin{align*}
(\nabla_X A)Y-(\nabla_Y A)X=&c\{\eta(X)\phi Y-\eta(Y)\phi X-2g(\phi X,Y)\xi\}\\
&+c\sum_{a=1}^3 \{\eta_a(X)\phi_a Y-\eta_a(Y)\phi_a X -2g(\phi_a X,Y)\xi_a\\&
+\eta_a(\phi X)\theta_a Y-\eta_a(\phi Y)\theta_a X\}.
\end{align*}
As  $M$ is a real hypersurface in $\hat M^m(c)$, by the Gauss equation, we have
\begin{align}\label{eqn:10}
R_\xi X=&\alpha AY-\eta(AY)A\xi+c\{X-\eta(X)\xi-\theta X\}		\notag\\
   &-c\sum^3_{a=1}\{\eta_a(X)\xi_a+3\eta_a(\phi X)\phi\xi_a\}.
\end{align}

We end this section with the following general result.

\begin{theorem}\label{T1}
Let $M$ be an almost contact metric manifold. The structure Jacobi operator $R_\xi$ is semi-parallel if and only if  $R_\xi=0$.
\end{theorem}
\begin{proof}
Suppose the structure Jacobi operator is semi-parallel. Then
\[
R(X,Y)R_\xi Z-R_\xi R(X,Y)Z=(R(X,Y)\cdot R_\xi)Z=0
\]
for any $X,Y,Z\in TM$. In particular, for $Y=Z=\xi$, we obtain $R_\xi^2X=0$. Since $R_\xi$ is self-adjoint, $R_\xi=0$.
The converse is trivial.
\end{proof}


\section{Proof of Theorem~\ref{T2}}
By virtue of Theorem~\ref{T1}, It suffices to show that the structure Jacobi operator cannot be identically zero.
Suppose to the contrary that $R_\xi=0$. Then by (\ref{eqn:10}), we have
\begin{align}\label{eqn:B-10}
\alpha AY-\eta(AY)A\xi+c\{Y-\eta(Y)\xi-\theta Y\}-c\sum^3_{a=1}\{\eta_a(Y)\xi_a+3\eta_a(\phi  Y)\phi\xi_a\}=0.
\end{align}

\begin{claim}
$\xi\notin\mathfrak D$ on an open dense subset of $M$.
\end{claim}
\begin{proof}
Suppose $\xi\in\mathfrak D$ on an open subset $G$ of $M$. For each $x\in G$, we have $\theta=0$.  It follows from
Lemma~\ref{lem:Aphixi_a} that $\phi\xi_1=0$ after putting  $Y=\phi\xi_1$ in  (\ref{eqn:B-10}); a contradiction. Hence we obtain the claim.
\end{proof}

Consider a point $x\in M$  on which $\xi\notin\mathfrak D$ on a neighborhood $G$ of $x$ in $M$.
We define subspaces $\mathcal F$, $\mathcal F(1)$ and $\mathcal F(-1)$ of $T_xM$  by
\begin{align*}
\mathcal F=\{X\in\mathcal H:  \eta(AX)=0\}, \quad
\mathcal F(1)=\mathcal F\cap\mathcal H(1), \quad \mathcal F(-1)=\mathcal F\cap\mathcal H(-1).
\end{align*}
It is clear that
\begin{align*}
AY=\lambda_\varepsilon Y, \quad \alpha\lambda_\varepsilon+c(1-\varepsilon ||\xi^\perp||)=0
\end{align*}
for any $Y\in\mathcal F(\varepsilon)$ and $\varepsilon\in\{1,-1\}$.
By (\ref{eqn:B-10}), we have
\begin{align*}
&(X\alpha)AY+\alpha(\nabla_XA)Y-\eta(AY)\big\{(\nabla_XA)\xi+A\nabla_X\xi\big\} \notag\\
&-\big\{g(\nabla_XA)Y,\xi)+g(AY,\nabla_X\xi)\big\}A\xi+c\{-g(\nabla_X\xi,Y)-\eta(Y)\nabla_X\xi-(\nabla_X\theta)Y\} \notag\\
&+c\sum^3_{a=1}\{-g(\nabla_X\xi_a,Y)\xi_a-\eta_a(Y)\nabla_X\xi_a+3g(\nabla_X\phi\xi_a,Y)\phi\xi_a
-3\eta_a(\phi Y)\nabla_x\phi\xi_a\}=0
\end{align*}
for any $X,Y\in T_xM$.
By using (\ref{eqn:contact}) and (\ref{eqn:global}),  the preceding equation becomes
\begin{align*}
&(X\alpha)AY+\alpha(\nabla_XA)Y-\eta(AY)\big\{(\nabla_XA)\xi+A\phi AX\big\} \\
&-\big\{g(\nabla_XA)Y,\xi)+g(A\phi AX,Y)\big\}A\xi		\\
&+c\{-g(\phi AX,Y)-\eta(Y)\phi AX	+4g(AX,Y)\phi\xi^\perp-4\eta^\perp(\phi Y)AX\} \notag\\
&+c\sum^3_{a=1}\{-g(\phi_aAX,Y)\xi_a-\eta_a(Y)\phi_aAX	\\
&+3g(\theta_aAX,Y)\phi\xi_a-3\eta_a(\phi Y)\theta_aAX-2\eta_a(\phi AX)\theta_aY\}=0.
\end{align*}
By the preceding equation and the Codazzi equation, we have
\begin{align}\label{eqn:B-100}
& 	 (X\alpha)AY- (Y\alpha)AX-\eta(AY)\big\{(\nabla_XA)\xi+A\phi AX\big\}
		+\eta(AX)\big\{(\nabla_YA)\xi+A\phi AY\big\}			\notag\\
&+c\big\{\eta(X)(\phi AY+\alpha\phi Y)-\eta(Y)(\phi AX+\alpha\phi X) -4\eta^\perp(\phi Y)AX+4\eta^\perp(\phi X)AY\big\}		\notag\\
&+2g(c(\phi+\phi^\perp)X-A\phi AX,Y)A\xi-cg(2\alpha\phi X+(\phi A+A\phi )X,Y)\xi	\notag\\
&+c\sum^3_{a=1}\{
			\eta_a(X)(\phi_a AY+\alpha\phi_a Y)-\eta_a(Y)(\phi_a AX+\alpha\phi_a X)	\notag\\
&			-\alpha\eta_a(\phi Y)\theta_a X+\alpha\eta_a(\phi X)\theta_a Y-3\eta_a(\phi Y)\theta_aAX+3\eta_a(\phi X)\theta_aAY		\notag\\
&+2\eta_a(\phi AY)\theta_aX-2\eta_a(\phi AX)\theta_aY-2\eta_a(X)\eta_a(\phi Y)A\xi+2\eta_a(Y)\eta_a(\phi X)A\xi\}			\notag\\
&-g(2\alpha\phi_a X+(\phi_a A+A\phi_a )X,Y)\xi_a+3g((\theta_a A-A\theta_a)X,Y)\phi\xi_a \}
	=0
\end{align}
for any $X,Y\in T_xM$.
By putting $X,Y\in\mathcal F$ in (\ref{eqn:B-100}), we have
\begin{align*}
& 	 (X\alpha)AY- (Y\alpha)AX			\notag\\
&+2g(c(\phi+\phi^\perp)X-A\phi AX,Y)A\xi-cg(2\alpha\phi X+(\phi A+A\phi )X,Y)\xi	\notag\\
&+c\sum^3_{a=1}\{	-g(2\alpha\phi_a X+(\phi_a A+A\phi_a )X,Y)\xi_a+3g((\theta_a A-A\theta_a)X,Y)\phi\xi_a\}
	=0.
\end{align*}
Since the first two terms are in $\mathcal H$ and the remaining are in $\mathcal H^\perp$, we have
\begin{align} \label{eqn:B-120}
&2g(c(\phi+\phi^\perp)X-A\phi AX,Y)A\xi-cg(2\alpha\phi X+(\phi A+A\phi )X,Y)\xi	\notag\\
&+c\sum^3_{a=1}\{	-g(2\alpha\phi_a X+(\phi_a A+A\phi_a )X,Y)\xi_a+3g((\theta_a A-A\theta_a)X,Y)\phi\xi_a\}
	=0
\end{align}
for any $X,Y\in\mathcal F$.
For any  $\varepsilon\in\{1,-1\}$,  we can further deduce from  (\ref{eqn:B-120}) that 
\begin{align}\label{eqn:B-130}
0=&2g(\phi X,Y)\{c(1-\varepsilon||\xi^\perp||)-\lambda_\varepsilon^2\}A\xi
					+g(\phi X,Y)c(2\alpha+2\lambda_\varepsilon)\left\{-\xi+\frac\varepsilon {||\xi^\perp||}\xi^\perp\right\} \notag\\
=&g(\phi X,Y)(\alpha+\lambda_\varepsilon)\left\{-\lambda_\varepsilon A\xi
					-c\left(\xi-\frac\varepsilon {||\xi^\perp||}\xi^\perp\right)\right\}	
\end{align}
 for any $X,Y\in\mathcal F(\varepsilon)$; and 
\begin{align}\label{eqn:B-140}
0=\sum^3_{a=1}\{-(2\alpha+\lambda_\varepsilon+\lambda_{-\varepsilon})g(\phi_aX,Y)\xi_a
					+3(\lambda_\varepsilon-\lambda_{-\varepsilon})g(\theta_aX,Y)\phi\xi_a\}
\end{align}
for any $X\in\mathcal F(\varepsilon)$ and $Y\in\mathcal F(-\varepsilon)$.

\begin{claim}\label{claim:A}
$\dim \mathcal H\geq 8$.
\end{claim}
\begin{proof}
Suppose $\dim \mathcal H=4$.
Since $\dim M=4m-1\geq 11$, we have $\xi\notin\mathfrak D^\perp$  or $0<||\xi^\perp||<1$.
Take a unit vector $V\in\mathcal F(1)$ such that
$\mathcal H(1)=\mathbb RV\oplus\mathbb R\phi_1V$ and $\mathcal H(-1)=\mathbb R\phi_2V\oplus\mathbb R\phi_3V$.

Substituting  $X=V$  and $Y=\phi_2 V$ in (\ref{eqn:B-140}), we obtain
\begin{align*}
0	=&-(2\alpha+\lambda_1+\lambda_{-1})\xi_2
					-3(\lambda_1-\lambda_{-1})\phi\xi_3	
\end{align*}
for any $Y\in\mathcal F(1)$. Since $\{\xi_a,\phi\xi_a\}_{a\in\{1,2,3\}}$ is linearly independent,
$-2c\alpha^{-1}||\xi^\perp||=\lambda_1-\lambda_{-1}=0$; a contradiction. Hence, the claim is obtained.
\end{proof}
According to Claim~\ref{claim:A},  there exists $X\in\mathcal F(\varepsilon)$ such that  $X\perp$ $\phi A\xi$.
Taking such a vector $X$ and $Y=\phi X$ in (\ref{eqn:B-130}), we obtain
\begin{align}\label{eqn:B-150}
(\alpha+\lambda_\varepsilon)\left\{\lambda_\varepsilon A\xi+c\left(\xi-\frac{\varepsilon}{||\xi^\perp||}\xi^\perp\right)\right\}=0
\end{align}
for any $\varepsilon\in\{1,-1\}$.

\begin{claim}\label{claim:B-I}
$||\xi^\perp||=1$ on $M$.
\end{claim}
\begin{proof}
Suppose $0<||\xi^\perp||<1$ on the open subset $G$ of $M$.
It clear that $\alpha+\lambda_1$ and $\alpha+\lambda_{-1}$ can not be both zero as
$\lambda_1-\lambda_{-1}=-2c\alpha^{-1}||\xi^\perp||\neq0$.
Fixed $\varepsilon\in\{1,-1\}$ such that $\alpha+\lambda_\varepsilon\neq0$.
It follows from (\ref{eqn:B-150}) that
\begin{align*}
\lambda_\varepsilon A\xi+c\left(\xi-\frac{\varepsilon}{||\xi^\perp||}\xi^\perp\right)=0.
\end{align*}
This imply that $A\xi\perp$ $\mathcal H$, so $\mathcal F(1)=\mathcal H(1)$ and $F(-1)=\mathcal H(-1)$.
Taking $X\in\mathcal H(1)$ and $Y=\phi_2X$ in (\ref{eqn:B-140}),  we can obtain a  contradiction by using a similar method as in the proof of Claim~\ref{claim:A}. Hence, $||\xi^\perp||=1$ at the point $x$. By the connectedness of $M$ and the continuity of $||\xi^\perp||$,
we conclude that $||\xi^\perp||=1$ on $M$.
\end{proof}

Since $||\xi^\perp||=1$ on $M$ or $\xi\in\mathfrak D^\perp$ everywhere, we have $\lambda_1=0$ and $\lambda_{-1}=-2c/\alpha$ ($=\lambda$, for simplicity).
Moreover, we have
\begin{align}\label{eqn:B-160}
-\sum^3_{a=1}\eta_a(\phi Y)\phi\xi_1=\sum^3_{a=1}\eta_a(Y)\xi_a-\eta(Y)\xi.
\end{align}
It follows from (\ref{eqn:B-10}) and (\ref{eqn:B-160}) that
\begin{align}\label{eqn:B-180}
\alpha AY-\eta(AY)A\xi+c\{Y-4\eta(Y)\xi-\theta Y\}+2c\sum^3_{a=1}\eta_a(Y)\xi_a=0.
\end{align}
On the other hand, we have
\begin{align*}
\sum^3_{a=1}&\{g(Y,\nabla_X\phi\xi_a)\phi\xi_a-\eta_a(\phi Y)\nabla_X\phi\xi_a\}	\\
=&\sum^3_{a=1}\{g(Y,\nabla_X\xi_a)\xi_a+\eta_a(Y)\nabla_X\xi_a\}-g(Y,\nabla_X\xi)\xi-\eta(Y)\nabla_X\xi.
\end{align*}
By using (\ref{eqn:contact}), we have
\begin{align}\label{eqn:B-200}
\sum^3_{a=1}&\{g(Y,\theta_aAX)\phi\xi_a-\eta_a(\phi Y)\theta_aAX\} \notag\\
=&\sum^3_{a=1}\{g(Y,\phi_aAX)\xi_a+\eta_a(Y)\phi_aAX\}-g(Y,\phi AX)\xi-\eta(Y)\phi AX.
\end{align}
\begin{claim}\label{claim:B-II-a}
$\lambda A\xi+2c\xi\neq0$ .
\end{claim}
\begin{proof}
Suppose  $\lambda A\xi+2c\xi=0$. Then $A\xi=\alpha\xi$ and $\alpha\lambda+2c=0$.
It follows that from (\ref{eqn:B-180}) that
\begin{align*}
AY=-\frac c\alpha(Y-\theta Y)+\frac{\alpha^2+4c}\alpha\eta(Y)\xi-\frac{2c}\alpha\sum^3_{a=1}\eta_a(Y)\xi_a.
\end{align*}
By \cite[Theorem 6.1]{lee-loo},  we obtain
\begin{align}\label{eqn:B-202}
\alpha^2+4c=0
\end{align}
and so $c<0$. Furthermore, we have either
\[
\frac{- c}\alpha=\frac{\sqrt{-2c}\tanh(\sqrt{-2c} r)}2, \quad \frac{-2c}\alpha=\sqrt{-2c}\coth(\sqrt{-2c} r), \quad r>0
\]
or
 \[
\frac {-c}\alpha=\frac{\sqrt{-2c}}2, \quad \frac{-2c}\alpha=\sqrt{-2c}.
\]
However, both cases contradict (\ref{eqn:B-202}). Accordingly, we obtain the claim.
\end{proof}

By using Claim~\ref{claim:B-II-a},  (\ref{eqn:B-150}) and (\ref{eqn:B-180}),
there exists a unit vector field $U$ tangent to $\mathcal H(-1)\oplus(\mathfrak D^\perp\ominus\mathbb R\xi)$ and functions
$\tau$, $\beta$ ($\beta\neq0$) on $M$ such that
\begin{align}\label{eqn:B-220}
\left.\begin{array}{ll}
A\xi=\alpha\xi+\beta U &\\
AU=\beta \xi+\tau U&\\
AY=\lambda Y, \quad  &(Y\in\mathcal F(-1))	\\
AX=0, \quad & (X\in\mathcal H(1))\\
\alpha\tau+\alpha^2=\beta^2	\\
\lambda+\alpha=0, \quad \alpha^2=2c	.
\end{array}\right\}
\end{align}
By using  (\ref{eqn:B-200}) and (\ref{eqn:B-220}),  (\ref{eqn:B-100}) is descended to
\begin{align}\label{eqn:B-240}
& 	-\eta(AY)\big\{\alpha\phi AX+(X\beta)U+\beta\nabla_XU\big\}
		+\eta(AX)\big\{\alpha\phi AY+(Y\beta)U+\beta\nabla_YU\big\}			\notag\\
&+c\big\{\eta(X)(4\phi AY+\alpha\phi Y)
                 -\eta(Y)(4\phi AX+\alpha\phi X)		\big\}		\notag\\
&+2g(c(\phi+\phi^\perp)X-A\phi AX,Y)A\xi-cg(2\alpha\phi X+4(\phi A+A\phi )X,Y)\xi	\notag\\
&+c\sum^3_{a=1}\{\eta_a(X)(-2\phi_a AY+\alpha\phi_a Y)
																	-\eta_a(Y)(-2\phi_a AX+\alpha\phi_a X)	\notag\\
&		 -\alpha\eta_a(\phi Y)\theta_a X
		 	+\alpha\eta_a(\phi X)\theta_a Y		
		           +2\eta_a(\phi AY)\theta_aX
							  -2\eta_a(\phi AX)\theta_aY		\notag\\
&		-2\eta_a(X)\eta_a(\phi Y)A\xi+2\eta_a(Y)\eta_a(\phi X)A\xi\}			\notag\\
&		+g(-2\alpha\phi_a X+(\phi_a A+A\phi_a )X,Y)\xi_a \}
	=0.
\end{align}
On the other hand, by the Codazzi equation, we obtain
\begin{align}\label{eqn:B-250}
cg((\phi+\phi^\perp)Y,Z)-2c\sum^3_{a=1}\eta_a(Y)\eta_a(\phi Z)=g((\nabla_\xi A)Y-(\nabla_YA)\xi,Z).
\end{align}
Substituting $Z=\xi$ in (\ref{eqn:B-250}) gives
\[
(\xi\beta)g(U,Y)+\beta g(\nabla_\xi U,Y)+4\beta\alpha g(\phi U,Y).
\]
Letting $Y=U$ in the preceding equation gives
\begin{align}	
\xi\beta=&0																								\label{eqn:B-260}\\
\nabla_\xi U+4\alpha\phi U=&0							\label{eqn:B-270}
\end{align}
Next,  with the help of (\ref{eqn:B-220}), after putting $Y\in\mathcal H(1)$ and $Z=U$ in (\ref{eqn:B-250}) gives
\begin{align}\label{eqn:B-280}
Y\beta=0
\end{align}
for any $Y\in\mathcal H(1)$.
By putting $X=\xi$ and $Y\perp$ $\xi$ in (\ref{eqn:B-240}), we have
\begin{align}\label{eqn:B-290}
&Y\beta+(\alpha^2+2\beta^2)g(Y,\phi U)+3cg(Y,\phi U-\phi^\perp U)
								+\sum^3_{a=1}\{-2\alpha\eta_a(\phi AY)\eta_a(U)				\notag\\
&-\alpha\beta g(\phi_a Y,U)\eta_a(U)-\alpha\beta	g(\theta_aY,U)\eta_a(\phi U)-2c\alpha\eta_a(U)\eta_a(\phi U)\}=0							
\end{align}
for any $Y\perp$ $\xi$.
In particular, for $Y\in\mathcal H(1)$, with the help of (\ref{eqn:B-280}), we obtain
\begin{align*}
0=-\sum^3_{a=1}\{g(\phi_a Y,U)\eta_a(U)-g(\theta_aY,U)\eta_a(\phi U)\}=2\sum^3_{a=1}\eta_a(U)g(\phi_aU,Y)
\end{align*}
for any $Y\in\mathcal H(1)$.

Denote by $U^-$ the $\mathcal H(-1)$-component of $U$.
If $U$ is tangent to $\mathfrak D^\perp$ on an open subset $G$ of $M$. Then for each $x\in G$, $\mathcal F(-1)=\mathcal H(-1)$
and so $A\mathfrak D^\perp\subset\mathfrak D^\perp$. By virtue of \cite[Theorem 3.6]{lee-loo}, $\xi$ is principal on $G$.
This contradicts Claim~\ref{claim:B-II-a}. Hence, we assume that $U^-\neq0$.
By putting $Y=\phi_bU^-$, $b\in\{2,3\}$ in the preceding equation, we obtain $\eta_2(U)=\eta_3(U)=0$.
Consequently, $U=U^-\in\mathcal H(-1)$ and $\phi U=\phi^\perp U$. These, together with (\ref{eqn:B-260}) and (\ref{eqn:B-290}) give
\begin{align*}
Y\beta=-(\alpha^2+2\beta^2)g(Y,\phi U)
\end{align*}
for any vector field $Y$ tangent to $M$.
It follows that
\begin{align*}
(XY-\nabla_XY)\beta=(\alpha^2+2\beta^2)\{4\beta g(X,\phi U)g(Y,\phi U)-g(Y,\nabla_X\phi U)\}.
\end{align*}
Hence
\[
g(Y,\nabla_X\phi U)-g(X,\nabla_YU)=0.
\]
By virtue of (\ref{eqn:B-220}) and (\ref{eqn:B-270}), after substituting $X=\xi$ and $Y=U$ in the preceding equation give
$4\alpha+\tau=0$. But then $\beta^2=\alpha\tau+\alpha^2=-3\alpha^2$; a contradiction.
This completes the proof.


\end{document}